\documentclass[12pt,twoside]{article}

\usepackage{amsmath,amssymb,amsthm,amscd,graphicx,color,accents}

\usepackage[colorlinks=true,linkcolor=blue,citecolor=blue,urlcolor=blue]{hyperref}

\usepackage{enumitem}
\makeatletter
\def\namedlabel#1#2{\begingroup
    #2%
    \def\@currentlabel{#2}%
    \phantomsection\label{#1}\endgroup
}
\makeatother

\numberwithin{equation}{section}

\theoremstyle{plain}
\newtheorem{theorem}{Theorem}[section]

\newtheorem{corollary}[theorem]{Corollary}
\newtheorem{remark}[theorem]{Remark}

\def \Re {\mathbb{R}}

\def \uu {\bm{u}}

\def \TT {\mb{T}}

\def \ww {\bm{w}}

\def\transp {{\boldsymbol\top}}

\def \E {\mathbb{E}}

\def \cc {\boldsymbol{c}}

\def \ee {\boldsymbol{e}}

\def \uu {\boldsymbol{u}}
\def \vv {\boldsymbol{v}}
\def \ww {\boldsymbol{w}}

\def \bzero {\boldsymbol{0}}

\def\ds {\hskip 1pt{\rm{d}}}

\def\bSigma {{\boldsymbol{\Sigma}}}
\def\bOmega {{\boldsymbol{\Omega}}}
\def\btheta {{\boldsymbol{\theta}}}
\def\bTheta {{\boldsymbol{\Theta}}}

\def \TT {\boldsymbol{T}}

\def \VV {\boldsymbol{V}}
\def \WW {{\boldsymbol{W}}}

\def \Mu {\boldsymbol{\mu}}

\def \Eta {\boldsymbol{\eta}}

\def\e{{\mathbb E}}

\def\Cov{{\mathrm{Cov}}}

\begin{document}

\title{\Large\textbf{Independence Properties of the Truncated Multivariate Elliptical Distributions}}

\author{
{Michael Levine,}\thanks{Department of Statistics, Purdue University, West Lafayette, IN 47907, U.S.A.}
\ {Donald Richards,}\thanks{Department of Statistics, Pennsylvania State University, University Park, PA 16802, U.S.A.}
\ {and Jianxi Su}\thanks{Department of Statistics, Purdue University, West Lafayette, IN 47907, U.S.A.
\endgraf
\ $^*$Corresponding author; e-mail address: mlevins@purdue.edu
}}

\date{\today}

\maketitle

\begin{abstract}
Truncated multivariate distributions arise extensively in econometric modelling, when non-negative random variables are intrinsic to the data-generation process.  and more broadly in censored and truncated regression models, simultaneous equations modelling, multivariate regression, and other areas.  In some applications, there arises the problem of characterizing truncated multivariate distributions through correlation and independence properties of sub-vectors.  In this paper, we characterize the truncated multivariate normal random vectors for which two complementary sub-vectors are mutually independent.  Further, we characterize the multivariate truncated elliptical distributions, proving that if two complementary sub-vectors are mutually independent then the distribution of the joint vector is truncated multivariate normal, as is the distribution of each sub-vector.  As an application, we apply the independence criterion to test the hypothesis of independence of the entrance examination scores and subsequent course averages achieved by a sample of university students; to do so, we verify the regularity conditions underpinning a classical theorem of Wilks on the asymptotic null distribution of the likelihood ratio test statistic.



\smallskip
\noindent
{\em Running head}: Truncated elliptical distributions.
\end{abstract}

\section{Introduction}
\label{sec:intro}

The truncated multivariate normal distributions are a family of distributions that have appeared in simultaneous equations modelling and multivariate regression \cite{Amemiya}, economics \cite{Heckman}, econometric models for auction theory \cite{Hong}, and other areas.  Consequently, there exists a wide literature on the properties of these distributions.

To define the truncated multivariate normal distributions, we recall the component-wise partial ordering on $p$-dimensional Euclidean space, $\Re^p$: For column vectors $\uu = (u_1,\ldots,u_p)^\transp$ and $\vv = (v_1,\ldots,v_p)^\transp$ in $\Re^p$ we write $\uu \ge \vv$ if $u_j \ge v_j$ for all $j=1,\ldots,p$.

Let $\Mu \in \Re^p$ and let $\bSigma$ be a $p \times p$ positive definite matrix.  For $\cc \in \Re^p$, we say that the random vector $\WW \in \Re^p$ has a \textit{truncated multivariate normal distribution, with truncation point $\cc$}, if the probability density function of $\WW$ is
\begin{equation}
\label{trnorm}
f(\ww;\Mu,\bSigma,\cc) = C \exp\left[-\tfrac12(\ww-\Mu)^\transp\bSigma^{-1}(\ww-\Mu)\right],
\qquad \ww \ge \cc,
\end{equation}
where $C$, the normalizing constant, is given by
$$
C^{-1} = \int_{\ww \ge \cc} \, \exp\left[-\tfrac12(\ww-\Mu)^\transp\bSigma^{-1}(\ww-\Mu)\right] \ds \ww.
$$
We write $\WW \sim N_p(\Mu,\bSigma,\cc)$ whenever $\WW$ has the density function \eqref{trnorm}.  Further, we denote the usual (untruncated) multivariate normal distribution by $N_p(\Mu,\bSigma)$.

Suppose that $\WW$, $\Mu$, and $\cc$ are partitioned into sub-vectors,
\begin{equation}
\label{vector_partition}
\WW = \begin{pmatrix}\WW_1 \\ \WW_2\end{pmatrix},
\qquad
\Mu = \begin{pmatrix}\Mu_1 \\ \Mu_2\end{pmatrix},
\qquad
\cc = \begin{pmatrix}\cc_1 \\ \cc_2\end{pmatrix},
\end{equation}
where $\WW_j$, $\Mu_j$, and $\cc_j$ each are of dimension $p_j$, $j=1,2$, with $p_1 + p_2 = p$.  Further, we partition $\bSigma$ so that
\begin{equation}
\label{Sigmadecomp}
\bSigma=\begin{pmatrix}
\bSigma_{11} & \bSigma_{12} \\
\bSigma_{21} & \bSigma_{22}
\end{pmatrix}
\end{equation}
where $\bSigma_{jk}$ is of order $p_j\times p_k$, for $j,k = 1,2$.  In a study of the correlation and independence properties of sub-vectors of truncated distributions, we show in Section \ref{sec:ellipticalcorr} that the uncorrelatedness of $\WW_1$ and $\WW_2$ cannot be characterized by the condition that $\bSigma_{12} = \bzero$.  Going beyond the study of the correlation properties of $\WW$, we prove in Section \ref{sec:normal} that the condition $\bSigma_{12} = \bzero$ is necessary and sufficient for $\WW_1$ and $\WW_2$ to be mutually independent; in particular, no restrictions are required on $\Mu$ or $\cc$.

More general than the truncated multivariate normal distributions are their elliptical counterparts.  For $\cc$ and $\Mu$ in $\Re^p$, and a positive definite matrix $\bSigma$, a random vector $\WW \in \Re^p$ is said to have a \textit{truncated elliptical distribution, with truncation point $\cc$}, if its probability density function is of the form
\begin{equation}
\label{trelli}
f(\ww;\Mu,\bSigma,\cc) = %|\bSigma|^{-1/2} \,
g\big((\ww-\Mu)^\transp\bSigma^{-1}(\ww-\Mu)\big), \qquad \ww \ge \cc,
\end{equation}
for a non-constant {\it generator} $g:[0,\infty) \to [0,\infty)$.  We write $(\WW_1,\WW_2)\sim E_{p}(\boldsymbol{\mu},\boldsymbol{\Sigma},g,\cc)$ with the untruncated counterpart being denoted by $E_{p}(\boldsymbol{\mu},\boldsymbol{\Sigma},g)$.  Examples of truncated elliptically contoured distributions are the truncated multivariate Student's $t$-distributions \cite{Ho_etal,Matos}.
% for algorithms for simulating these distributions and for formulas for calculating various moments.
We prove in Section \ref{sec:elliptical} that if $(\WW_1,\WW_2)\sim E_{p}(\boldsymbol{\mu},\boldsymbol{\Sigma},g,\cc)$ then, under certain regularity conditions on the generator $g$,  a necessary and sufficient condition that $\WW_1$ and $\WW_2$ be independent is that $\bSigma_{12} = \bzero$.   Here again, no conditions are required on $\Mu$ or $\cc$; moreover, we verify that the stated regularity conditions on $g$ are mild since they hold for many familiar elliptical distributions.

In Section \ref{sec:application}, we consider for illustrative purposes an application of the criterion derived in Section \ref{sec:normal} to testing the hypothesis of independence of the sub-vectors $\WW_1$ and $\WW_2$.  We obtain from the classical theorem of Wilks \cite{Hogg} the asymptotic null distribution of the likelihood ratio test statistic, and we provide an application to a data set given by Cohen \cite{Cohen} on the entrance examination scores and subsequent course averages achieved by a large sample of university students.

\section{Correlation properties of truncated elliptical distributions}
\label{sec:ellipticalcorr}

In this section we show, first, that the correlation structure of a multivariate elliptical distribution does not describe the correlation structure of its truncated version. More precisely, even if a particular multivariate elliptical distribution possesses an identity correlation matrix, this fact is not equivalent to the lack of correlation between components of the truncated version of that multivariate elliptical distribution.

We will demonstrate our claim using the bivariate case.  Starting with elliptically distributed random variables $(X_1,X_2)^\transp \sim E_2(\boldsymbol{\mu},\boldsymbol{\Sigma},g)$, set
\begin{equation*}
\label{Sigma2by2}
\bSigma = \begin{pmatrix}
1 & \rho \\ \rho & 1
\end{pmatrix},
\end{equation*}
without loss of generality, where $|\rho| < 1$.  Let $(W_1,W_2) = (X_1,X_2) | \{X_1 \ge c_1, X_2\ge c_2\}$ be the version of $(X_1,X_2)$ that is truncated at $\cc = (c_1,c_2)^\transp$.  For simplicity, consider the case in which $\cc = \boldsymbol{\mu}$, so that $(W_1,W_2)^\transp \sim E_{2}(\boldsymbol{\mu},\boldsymbol{\Sigma},g,\cc).$ We will now show that uncorrelatedness between $W_1$ and $W_2$ is not equivalent to $\rho=0$.

At the outset, let us recall from \cite{Fang} a stochastic representation for elliptically distributed random variables:
\[
\begin{pmatrix}X_1 \\ X_2\end{pmatrix} \overset{d}{=} R\ \boldsymbol{\Sigma}^{1/2} \begin{pmatrix}U_1 \\ U_2\end{pmatrix} + \boldsymbol{\mu},
\]
where $(U_1,U_2)^\transp$ is distributed uniformly over the unit circle, and the generating random variable $R$ has the density function $f(r)=2\pi r g(r^2),$ $r>0$.  Define
$$
U_1^*=U_1, \quad U^*_2=\rho U_1+(1-\rho^2)^{1/2}\, U_2;
$$
then, $\boldsymbol{\Sigma}^{1/2}\ (U_1,U_2)^\transp = (U_1^*,U^*_2)^\transp$, and
%the covariance between $W_1$ and $W_2$ is given by
\begin{align*}
\Cov(W_1,W_2) &= \e\big(\e[W_1W_2|R]\big) - \e\big(\e[W_1|R]\big) \ \e\big(\e[W_2|R]\big) \\
  &= \e(R^2) \, \e[U_1^*U_2^*|U_1^*>0,U_2^*>0] - (\e\, R)^2 \, \prod_{i=1}^2\e[U_i^*|U_1^*>0,U_2^*>0].
\end{align*}
To calculate these conditional expectations, we transform $(U_1^*,U_2^*)$ to polar coordinates,
$$
U_1^*= \cos \Psi, \quad U_2^* = \rho \cos \Psi + (1-\rho^2)^{1/2} \sin \Psi,
$$
where the random variable $\Psi$ is uniformly distributed on the interval $(-\pi,\pi)$.  Letting $\psi^* =\tan^{-1}(-(1-\rho^2)^{-1/2}\rho)$, we obtain
\begin{align*}
  \e[U_1^*U_2^*&|U_1^*>0,U_2^*>0] \\
  &= \e\big[\rho \cos^2 \Psi + (1-\rho^2)^{1/2} \sin \Psi \cos \Psi \, \big| \, \Psi \in ( \psi^*,\pi/2) \big]\\
  &= (\tfrac12\pi - \psi^*)^{-1} \left[\tfrac14 \rho (\pi-\sin 2\psi^*)
  -\tfrac12 \rho\psi + \tfrac14 (1-\rho^2)^{1/2}(1 + \cos 2\psi^*) \right] \equiv h_1(\rho).
\end{align*}
Similarly,
\begin{align*}
  \e[U_1^*|U_1^*>0,U_2^*>0] &= \e[\cos \Psi \, \big| \, \Psi \in ( \psi^*,\pi/2) ] \\
  &=  (\tfrac12\pi - \psi^*)^{-1} \left[1-\sin \psi^*  \right] \equiv h_2(\rho),
\end{align*}
and
\begin{align*}
\e[U_2^*|U_1^*>0,U_2^*>0] &= \e[\rho \cos \Psi + (1-\rho^2)^{1/2}\sin \Psi \, \big| \, \Psi \in (\psi^*,\pi/2) ]\\
  &= (\tfrac12\pi - \psi^*)^{-1} \left[\rho(1-\sin \psi^*) + (1-\rho^2)^{1/2} \cos \psi^* \right] \equiv h_3(\rho).
\end{align*}

In summary, we have obtained
\begin{eqnarray}
\label{eqn:bivariate-cov}
\Cov(W_1,W_2)= \e[R^2] h_1(\rho) - [\e \, R]^2 h_2(\rho) h_3(\rho).
\end{eqnarray}
Note that $\rho=0$ implies $\psi^* = 0$.  Hence, $h_1(0)=1/\pi$ and $h_2(0)=h_3(0)=2/\pi$.
% it is evident that $\rho = 0$ does not imply that $W_1$ and $W_2$ are uncorrelated.

We remark that uncorrelatedness cannot be characterized for all elliptical truncated distributions through the condition $\rho=0$.  Consider, for instance, the truncated bivariate Student's $t$-distribution with degrees-of-freedom $\tau>0$, where the associated generating variable $R$ has the density function that is proportional to
$
%\label{eqn:t-density}
%f(r) \propto r
(1+\tau^{-1} r^2)^{-(\tau+2)/2}$,
%\end{eqnarray}
$r>0$; this density corresponds to the generalized beta distribution of the second kind \cite{mcd}.  It is straightforward to deduce that
\[
\frac{\e[R^2]}{[\e \, R]^2} = \frac{4 \Gamma(\tau/2)\Gamma((\tau-2)/2)}{\pi [\Gamma((\tau-1)/2)]^2},
\]
$\tau > 2$.  Noting that the gamma function $\Gamma(\cdot)$ is strictly log-convex \cite{Artin}, we have
\[
\log(\Gamma((\tau-1)/2) < \left[ \log(\Gamma(\tau/2))+\log(\Gamma((\tau-2)/2)\right]/2,
\]
$\tau > 2$, equivalently ${\e[R^2]}/{[\e\, R]^2} > 4/\pi$.  By Equation (\ref{eqn:bivariate-cov}), $\Cov(W_1,W_2)>0$; hence, for the truncated bivariate Student's $t$-distributions with truncation points equal to the means, the condition $\rho=0$ implies that $W_1$ and $W_2$ are positively correlated.

We remark that for the above example, uncorrelatedness holds in a limiting sense as $\tau \rightarrow \infty$; in that case, ${\e[R^2]}/{[\e\, R]^2} \rightarrow \pi/4$ and hence $\Cov(W_1,W_2)\rightarrow 0$.  This limiting case corresponds to the truncated bivariate normal distributions, which we treat in the next section.

On the other hand, for given $\rho \neq 0$, we can apply Equation (\ref{eqn:bivariate-cov}) to construct a plethora of truncated elliptical distributions that are uncorrelated.  For the sake of illustration, suppose that $\rho = -1/\sqrt{2}$; then $\psi^* = \pi/4$ and
\[
h_1(-1/\sqrt{2}) = \frac{\sqrt{2}(4-\pi)}{4\pi},\ \ h_2(-1/\sqrt{2}) = h_3(-1/\sqrt{2})=\frac{4(2-\sqrt{2})}{2\pi}.
\]
Therefore, for any truncated elliptical distributions whose generating variable satisfies
\begin{eqnarray}
\label{eqn:zero-corr}
b:= \frac{\e[R^2]}{[\e \, R]^2}=\frac{h_2(-1/\sqrt{2})\ h_3(-1/\sqrt{2})}{h_1(-1/\sqrt{2})}=\frac{16(3\sqrt{2}-4)}{\pi(4-\pi)} \approx 1.44,
\end{eqnarray}
the variables $W_1$ and $W_2$ are uncorrelated.  For example, if $R$ follows a gamma distribution with shape parameter $(b-1)^{-1} \approx 2.27$ and any positive scale parameter, then Equation \eqref{eqn:zero-corr} can be satisfied.

We have now shown that even in the bivariate case and for the special case in which the truncation vector $\cc$ equals the mean $\boldsymbol{\mu}$, the truncated elliptical distributions do not inherit the correlation property of the untruncated elliptical distributions.  On the one hand, it is possible that $\rho=0$ can lead to positively correlated $W_1$ and $W_2$, as we have seen from the example on the truncated Student's $t$-distributions.  On the other hand, there exist elliptical distributions with $\rho<0$ such that the components of their truncated versions are uncorrelated.

\section{The multivariate normal case}
\label{sec:normal}

%Independence is, of course, a stronger requirement than uncorrelatedness. Recall that for multivariate normally distributed vectors, the components are independent if and only if the correlation matrix is an identity matrix.  For multivariate elliptically distributed vectors, if the components are independent then the distribution is necessarily normal.

Throughout the rest of the paper, we denote by $\bzero$ any zero matrix or vector, irrespective of the dimension. In this section, we prove that the independence property of multivariate normal distributions can be carried over to their truncated counterparts.
%In this section, we establish the following result.

\begin{theorem}
\label{independence_trnorm}
Suppose that the random vector $\WW \sim N_p(\Mu,\bSigma,\cc)$ is decomposed as in \eqref{vector_partition}.  Then $\WW_1$ and $\WW_2$ are mutually independent if and only if $\bSigma_{12} = \bzero$.
\end{theorem}

We remark that this result was stated in \cite[p.~214]{Horrace03}.  However, an inspection of the purported proof \cite[p.~218]{Horrace03} reveals that the `if' part of the result solely was established, so the converse assertion has remained open.  Unlike the classical untruncated normal distribution, the matrix $\bSigma$ is not the covariance matrix of $\WW$, so it is surprising that the independence of $\WW_1$ and $\WW_2$ is characterized by the condition $\bSigma_{12} = \bzero$.

\medskip

\noindent{\it Proof of Theorem \ref{independence_trnorm}}.
First, we note that
\begin{equation}
\label{wwgecc}
\{\ww \in \Re^p: \ww \ge \cc\} \equiv  \{\ww_1 \in \Re^{p_1}, \ww_2 \in \Re^{p_2}: \ww_1 \ge \cc_1, \ww_2 \ge \cc_2\}
\end{equation}

Now suppose that $\bSigma_{12} = \bzero$.   Then it is evident from \eqref{trnorm}, \eqref{Sigmadecomp}, and \eqref{wwgecc} that the density of $\WW$ reduces to a product of two terms corresponding to the distributions $N_{p_1}(\Mu_1,\bSigma_{11},\cc_1)$ and $N_{p_2}(\Mu_2,\bSigma_{22},\cc_2)$.  Consequently, $\WW_1$ and $\WW_2$ are mutually independent, and $\WW_j \sim N_{p_j}(\Mu_j,\bSigma_{jj},\cc_j)$, $j=1,2$.

Conversely, suppose that $\WW_{1}$ and $\WW_2$ are mutually independent.  For $\WW \sim N_p(\Mu,\bSigma,\cc)$, it is evident that $\WW-\cc \sim N_p(\Mu-\cc,\bSigma,\bzero)$.  Since $\WW_1$ and $\WW_2$ are mutually independent if and only if $\WW_1 - \cc_1$ and $\WW_2 - \cc_2$ are mutually independent then we can assume, with no loss of generality, that $\cc = \bzero$.

Thus, for $\WW \sim N_p(\Mu,\bSigma,\bzero)$, suppose that $\WW_1$ is independent of $\WW_2$.  By a well-known quadratic form decomposition (Anderson \cite[p.~638]{Anderson}), we have
\begin{equation}
\begin{aligned}
\label{quadraticdecomposition}
(&\ww-\Mu)^\transp\bSigma^{-1}(\ww-\Mu) \\
&= (\ww_1-\Mu_1)^\transp\bSigma_{11}^{-1}(\ww_1-\Mu_1) \\
& \quad + \big(\ww_2 - \Mu_2 - \bSigma_{21} \bSigma_{11}^{-1}(\ww_1-\Mu_1)\big)^\transp \bSigma_{22 \cdot 1}^{-1} \big(\ww_2 - \Mu_2 - \bSigma_{21} \bSigma_{11}^{-1}(\ww_1-\Mu_1)\big),
\end{aligned}
\end{equation}
where $\bSigma_{22 \cdot 1} = \bSigma_{22} - \bSigma_{21} \bSigma_{11}^{-1} \bSigma_{12}$.
Applying this decomposition to the density function \eqref{trnorm}, we find that in order to calculate the marginal density of $\WW_1$ it is necessary to consider the integral
\begin{multline}
\label{quadraticformintegral}
\int_{\ww_2 \ge \bzero} \exp \big[-\tfrac12
\big(\ww_2 - \Mu_2 - \bSigma_{21} \bSigma_{11}^{-1}(\ww_1-\Mu_1)\big)^\transp \\
\times \bSigma_{22 \cdot 1}^{-1} \big(\ww_2 - \Mu_2 - \bSigma_{21}\bSigma_{11}^{-1}(\ww_1-\Mu_1)\big)\big] \ds\ww_2.
\end{multline}
For fixed $\ww_1$, suppose that $\VV$ is a $p_2$-dimensional multivariate normal random vector with $\VV \sim N_{p_2}\big(\Mu_2 + \bSigma_{21} \bSigma_{11}^{-1}(\ww_1-\Mu_1),\bSigma_{22 \cdot 1}\big)$.  Then the integral \eqref{quadraticformintegral} equals
$$
(2\pi)^{p_2/2} \, (\det \bSigma_{22 \cdot 1})^{1/2} \, P(\VV \ge \bzero),
$$
Let $\VV_0 = \VV - \Mu_2 - \bSigma_{21} \bSigma_{11}^{-1}(\ww_1-\Mu_1) \sim N_{p_2}(\bzero,\bSigma_{22 \cdot 1})$; then,
\begin{align*}
P(\VV \ge \bzero) &= P\big(\VV - \Mu_2 - \bSigma_{21} \bSigma_{11}^{-1} (\ww_1-\Mu_1) \ge -\Mu_2 - \bSigma_{21} \bSigma_{11}^{-1} (\ww_1-\Mu_1)\big) \\
&= P\big(\VV_0 \ge - \Mu_2 - \bSigma_{21} \bSigma_{11}^{-1} (\ww_1-\Mu_1)\big).
\end{align*}
Since $\VV_0$ has the same distribution as $-\VV_0$ then it follows that
\begin{align*}
P(\VV \ge \bzero) &= P\big(-\VV_0 \ge - \Mu_2 -\bSigma_{21} \bSigma_{11}^{-1} (\ww_1-\Mu_1)\big) \\
&= P(\VV_0 \le \Mu_2 + \bSigma_{21} \bSigma_{11}^{-1} (\ww_1-\Mu_1)\big),
\end{align*}
and we denote this probability by $\Phi_{p_2}\big(\Mu_2 + \bSigma_{21} \bSigma_{11}^{-1} (\ww_1-\Mu_1),\bSigma_{22 \cdot 1}\big)$.

Therefore, the marginal density function of $\WW_1$ is
\begin{multline}
\label{W1marginalpdf}
f_{\WW_1}(\ww_1) = C \, (2\pi)^{p_2/2} \, (\det \bSigma_{22 \cdot 1})^{1/2}
\exp\big[-\tfrac12 (\ww_1-\Mu_1)^\transp\bSigma_{11}^{-1} (\ww_1-\Mu_1)\big] \\
\times \Phi_{p_2}\big(\Mu_2 + \bSigma_{21} \bSigma_{11}^{-1} (\ww_1-\Mu_1),\bSigma_{22 \cdot 1}\big),
\end{multline}
$\ww_1 \ge \bzero$.  It now follows from \eqref{trnorm}, \eqref{W1marginalpdf}, and the quadratic form decomposition \eqref{quadraticdecomposition}, that the conditional density function of $\WW_2$, given $\WW_1=\ww_1$, is
\begin{align*}
&f_{\WW_2|\WW_1=\ww_1}(\ww_2) \\
&= \frac{(2\pi)^{-p_2/2} \, (\det \bSigma_{22 \cdot 1})^{-1/2}}{\Phi_{p_2}\big(\Mu_2 + \bSigma_{21} \bSigma_{11}^{-1} (\ww_1-\Mu_1),\bSigma_{22 \cdot 1}\big)}
\frac{\exp\left[-\tfrac12(\ww-\Mu)^\transp\bSigma^{-1}
(\ww-\Mu)\right]}{\exp\left(-\tfrac12 (\ww_1-\Mu_1)^\transp\bSigma_{11}^{-1} (\ww_1-\Mu_1)\right)} \\
&= \frac{(2\pi)^{-p_2/2} \, (\det \bSigma_{22 \cdot 1})^{-1/2}}{\Phi_{p_2}\big(\Mu_2 + \bSigma_{21} \bSigma_{11}^{-1} (\ww_1-\Mu_1),\bSigma_{22 \cdot 1}\big)} \\
& \quad\times \exp \big[-\tfrac12
\big(\ww_2 - \Mu_2 - \bSigma_{21} \bSigma_{11}^{-1} (\ww_1-\Mu_1)\big)^\transp \bSigma_{22 \cdot 1}^{-1} \big(\ww_2 - \Mu_2 - \bSigma_{21} \bSigma_{11}^{-1} (\ww_1-\Mu_1)\big)\big],
\end{align*}
$\ww_1 \ge \bzero$, $\ww_2 \ge \bzero$.

Since $\WW_1$ is independent of $\WW_2$ then $f_{\WW_2|\WW_1=\ww_1}$ is constant in $\ww_1$.  Therefore,
$$
f_{\WW_2|\WW_1=\ww_1}(\ww_2)\equiv \lim_{\ww_1 \to \bzero} f_{\WW_2|\WW_1=\ww_1}(\ww_2),
$$
$\ww_1 \ge \bzero$, $\ww_2 \ge \bzero$, so we obtain
\begin{multline*}
\frac{\exp \big[-\tfrac12
\big(\ww_2 - \Mu_2 - \bSigma_{21} \bSigma_{11}^{-1} (\ww_1-\Mu_1)\big)^\transp \bSigma_{22 \cdot 1}^{-1} \big(\ww_2 - \Mu_2 - \bSigma_{21} \bSigma_{11}^{-1} (\ww_1-\Mu_1)\big)\big]}{(2\pi)^{p_2/2} \, (\det \bSigma_{22 \cdot 1})^{1/2}  \, \Phi_{p_2}\big(\Mu_2+\bSigma_{21} \bSigma_{11}^{-1} (\ww_1-\Mu_1),\bSigma_{22 \cdot 1}\big)} \\
\equiv \frac{\exp \big[-\tfrac12
\big(\ww_2 - \Mu_2 + \bSigma_{21} \bSigma_{11}^{-1} \Mu_1\big)^\transp \bSigma_{22 \cdot 1}^{-1} \big(\ww_2 - \Mu_2 + \bSigma_{21} \bSigma_{11}^{-1} \Mu_1\big)\big]}{(2\pi)^{p_2/2} \, (\det \bSigma_{22 \cdot 1})^{1/2}  \, \Phi_{p_2}(\Mu_2,\bSigma_{22 \cdot 1})},
\end{multline*}
$\ww_1 \ge \bzero$, $\ww_2 \ge \bzero$.  Cancelling common terms, we obtain
\begin{align*}
&\frac{\Phi_{p_2}\big(\Mu_2+\bSigma_{21} \bSigma_{11}^{-1} (\ww_1-\Mu_1),\bSigma_{22 \cdot 1}\big)}{\Phi_{p_2}(\Mu_2,\bSigma_{22 \cdot 1})} \\
& \qquad \equiv
\exp\big[-\tfrac12 \big(\ww_2 - \Mu_2 - \bSigma_{21} \bSigma_{11}^{-1} (\ww_1-\Mu_1)\big)^\transp \bSigma_{22 \cdot 1}^{-1} \big(\ww_2 - \Mu_2 - \bSigma_{21} \bSigma_{11}^{-1} (\ww_1-\Mu_1)\big) \\
& \qquad\qquad\qquad\quad + \tfrac12
\big(\ww_2 - \Mu_2 + \bSigma_{21} \bSigma_{11}^{-1} \Mu_1\big)^\transp \bSigma_{22 \cdot 1}^{-1} \big(\ww_2 - \Mu_2 + \bSigma_{21} \bSigma_{11}^{-1} \Mu_1\big)\big] \\
& \qquad = \exp\big[\tfrac12 \ww_1^\transp \bSigma_{11}^{-1} \bSigma_{12} \bSigma_{22 \cdot 1}^{-1} \big(2(\ww_2-\Mu_2) - \bSigma_{21} \bSigma_{11}^{-1} (\ww_1 - 2\Mu_1)\big)\big].
\end{align*}
Note that the left-hand side contains no term in $\ww_2$, whereas the right-hand side does.  Therefore, for all $\ww_1$, the coefficient of $\ww_2$ on the right-hand side necessarily is the zero vector; this can be proved by taking the logarithm of both sides and then calculating the gradient with respect to $\ww_2$.

Hence, $\ww_1^\transp \bSigma_{11}^{-1} \bSigma_{12} \bSigma_{22 \cdot 1}^{-1} \equiv \bzero$.  Since this holds for all $\ww_1 \ge \bzero$ then we obtain $\bSigma_{11}^{-1} \bSigma_{12} \bSigma_{22 \cdot 1}^{-1} = \bzero$.  As $\bSigma_{11}$ and $\bSigma_{22 \cdot 1}$ are non-singular, it follows that $\bSigma_{12} = \bzero$.
$\quad\qed$

\begin{remark}
\label{remark_W_2_untruncated}
{\rm
We remark that since the condition $\bSigma_{12} = \bzero$, which is necessary and sufficient for $\WW_1$ and $\WW_2$ to be mutually independent, requires no restrictions on $\cc$, then the same result holds if we let $\cc_2 \to -\infty$.  Consequently, Theorem \ref{independence_trnorm} remains valid if $\WW_1$ is truncated and $\WW_2$ is untruncated.
}\end{remark}

\section{The elliptical case}
\label{sec:elliptical}

In the elliptical case, as in the normal case, we may assume with no loss of generality, that the truncation point is $\cc = \bzero$.  Suppose that $\WW = (\WW_1,\WW_2)$ has a  truncated elliptical distribution with density function \eqref{trelli}.  Let
\[
Q(\ww) \equiv Q(\ww_1,\ww_2) = (\ww-\Mu)^\transp\bSigma^{-1}(\ww-\Mu),
\]
so the joint p.d.f. of $\WW$ is $g(Q(\ww_1,\ww_2))$.  In characterizing the distribution of $\WW$ through the independence of $\WW_1$ and $\WW_2$, we will require the following regularity conditions on the generator $g$:
\begin{enumerate}
\setlength\itemsep{-0.1em}
\item[(\namedlabel{R1}{R1})]
$g(t) > 0$ for all $t \ge 0$, $g$ is everywhere differentiable on $(0,\infty)$, and its derivative $g'$ is continuous.
\item[(\namedlabel{R2}{R2})]
The support of $g'$, {\it i.e.}, ${\rm supp}(g') = \{t > 0: g'(t) \neq 0\}$, is dense in $(0,\infty)$.
\item[(\namedlabel{R3}{R3})]
As $t\rightarrow \infty$, $\ds(\log g(t^2))/\ds t$ either tends to zero or diverges.
\end{enumerate}
We remark that these conditions appear to be mild as almost all of the commonly-used elliptical density functions that are described in \cite[Chapter 3]{Fang} satisfy (\ref{R1})-(\ref{R3}), an exception being the Kotz distribution with power parameter in the exponential term equal to $1/2.$

Now we establish as a consequence of Theorem \ref{independence_trnorm} a result that, under the regularity conditions (\ref{R1})-(\ref{R3}), a truncated multivariate elliptical distribution whose component vectors are independent can only be a truncated multivariate normal distribution.

\begin{corollary}
\label{ell}
Suppose that the generator $g$ satisfies the regularity conditions (\ref{R1})-(\ref{R3}).  Then $\WW_1$ and $\WW_2$ are independent if and only if $\WW$ has a truncated multivariate normal distribution with $\bSigma_{12}=\bzero$.
\end{corollary}

\begin{proof}
If $\WW$ has a truncated multivariate normal distribution with $\bSigma_{12}=\bzero$ then we have seen before that $\WW_1$ and $\WW_2$ are mutually independent, so we need only show the converse.

By integration, we obtain the marginal density function of $\WW_2$ as
$$
f_{\WW_2}(\ww_2) = \int_{\vv \ge \bzero} g(Q(\vv,\ww_2)) \ds \vv,
$$
and then the conditional density of $\WW_1$, given $\WW_2 = \ww_2$, is
\begin{equation}
\label{W1W2conditionalpdf}
\frac{f_\WW(\ww_1,\ww_2)}{f_{\WW_2}(\ww_2)} = \frac{g(Q(\ww_1,\ww_2))}{\int_{\vv \ge \bzero} g(Q(\vv,\ww_2)) \ds \vv}.
\end{equation}

Note that $\WW_1$ and $\WW_2$ are independent if and only if the conditional density function, \eqref{W1W2conditionalpdf}, of $\WW_1$, given $\WW_2 = \ww_2$, is constant in $\ww_2$.  By taking logarithms in \eqref{W1W2conditionalpdf} and then applying the gradient operator $\nabla_{\ww_2} = (\partial/\partial w_{p_1+1},\ldots,\partial/\partial w_p)^\transp$, we find that a necessary and sufficient condition for $\WW_1$ and $\WW_2$ to be independent is that
\begin{equation}
\label{W1W2_nasc_indep}
\frac{g'(Q(\ww_1,\ww_2))}{g(Q(\ww_1,\ww_2))} [\nabla_{\ww_2} Q(\ww_1,\ww_2)] = \frac{\int_{\vv \ge \bzero} g'(Q(\vv,\ww_2)) [\nabla_{\ww_2} Q(\vv,\ww_2)] \ds \vv}{\int_{\vv \ge \bzero} g(Q(\vv,\ww_2)) \ds \vv}
\end{equation}
for all $\ww_1 \ge \bzero$, $\ww_2 \ge \bzero$.  By \eqref{quadraticdecomposition},
$$
\nabla_{\ww_2} Q(\ww_1,\ww_2) = 2 \bSigma_{22 \cdot 1}^{-1} \big(\ww_2 - \Mu_2 - \bSigma_{21} \bSigma_{11}^{-1}(\ww_1-\Mu_1)\big);
$$
substituting this result in \eqref{W1W2_nasc_indep}, we find that a necessary and sufficient condition for independence is
\begin{multline*}
\frac{g'(Q(\ww_1,\ww_2))}{g(Q(\ww_1,\ww_2))} \bSigma_{22 \cdot 1}^{-1} \big(\ww_2 - \Mu_2 - \bSigma_{21} \bSigma_{11}^{-1}(\ww_1-\Mu_1)\big) \\
= \frac{\int_{\vv \ge \bzero} g'(Q(\vv,\ww_2)) \bSigma_{22 \cdot 1}^{-1} \big(\ww_2 - \Mu_2 - \bSigma_{21} \bSigma_{11}^{-1}(\vv-\Mu_1)\big) \ds \vv}{\int_{\vv \ge \bzero} g(Q(\vv,\ww_2)) \ds \vv},
\end{multline*}
$\ww_1 \ge \bzero$, $\ww_2 \ge \bzero$.  Cancelling $\bSigma_{22 \cdot 1}^{-1}$ on both sides of the latter equation, we obtain
\begin{multline}
\label{W1W2_nasc_indep2}
\frac{g'(Q(\ww_1,\ww_2))}{g(Q(\ww_1,\ww_2))} \big(\ww_2 - \Mu_2 - \bSigma_{21} \bSigma_{11}^{-1}(\ww_1-\Mu_1)\big) \\
= \frac{\int_{\vv \ge \bzero} g'(Q(\vv,\ww_2)) \big(\ww_2 - \Mu_2 - \bSigma_{21} \bSigma_{11}^{-1}(\vv-\Mu_1)\big) \ds \vv}{\int_{\vv \ge \bzero} g(Q(\vv,\ww_2)) \ds \vv},
\end{multline}
$\ww_1 \ge \bzero$, $\ww_2 \ge \bzero$.

%Denote by $|\Mu_2|$ the vector whose components are the absolute values of the components of $\Mu_2$; then, $|\Mu_2| \ge \bzero$.
Let $\Eta \ge \bzero$ be such that $\Eta \neq \Mu_2 - \bSigma_{21}\bSigma_{11}^{-1} \Mu_1$.  Evaluating both sides of \eqref{W1W2_nasc_indep2} at $\ww_2 = \Eta$, we obtain
\begin{multline}
\label{W1W2_nasc_indep3}
\frac{g'(Q(\ww_1,\Eta))}{g(Q(\ww_1,\Eta))} \big(\Eta - \Mu_2 - \bSigma_{21} \bSigma_{11}^{-1}(\ww_1-\Mu_1)\big) \\
= \frac{\int_{\vv \ge \bzero} g'(Q(\vv,\Eta)) \big(\Eta - \Mu_2 - \bSigma_{21} \bSigma_{11}^{-1}(\vv-\Mu_1)\big) \ds \vv}{\int_{\vv \ge \bzero} g(Q(\vv,\Eta)) \ds \vv},
\end{multline}
equivalently,
\begin{equation}
\label{gprime_condition}
\frac{g'(Q(\ww_1,\Eta))}{g(Q(\ww_1,\Eta))} (\bSigma_{21} \bSigma_{11}^{-1}\ww_1 + \cc_2) = \cc_1,
\end{equation}
for all $\ww_1 \ge \bzero$, where $\cc_2 = \Mu_2  -\Eta - \bSigma_{21} \bSigma^{-1}_{11} \Mu_1$ and $\cc_1$ is a $p_2 \times 1$ constant vector.

We also have $\|\cc_1\| < \infty$; otherwise, the left-hand side of \eqref{W1W2_nasc_indep3} is infinite for all $\ww_1 \ge \bzero$, and then it follows that $|g'(Q(\ww_1,\Eta))|$ is infinite for all $\ww_1 \ge \bzero$.  This implies that $g$ is unbounded everywhere, which is not possible since $g$ generates a density function.

Suppose that $\cc_1 = \bzero$; then, by \eqref{gprime_condition}, $g'(Q(\ww_1,\Eta)) = 0$ or $\bSigma_{21} \bSigma^{-1}_{11} \ww_1 + \cc_2 = \bzero$ for all $\ww_1  \ge \bzero$.  If $g'(Q(\ww_1,\Eta)) = 0$ for all $\ww_1 \ge \bzero$ then it follows that $g$ is a constant function; however, by (\ref{R2}), the support of $g'$ is dense, therefore $g$ cannot generate a density.  Also, by construction, $\cc_2 \neq \bzero$, so $\bSigma_{21} \bSigma^{-1}_{11} \ww_1 + \cc_2 \neq \bzero$ for all $\ww_1 \ge \bzero$.  Therefore, we have shown by contradiction that $\cc_1 \neq \bzero$.

Now suppose that $\bSigma_{12} \neq \bzero$.  Since $\bSigma_{11}$ is positive definite then $\bSigma_{11}^{-1} \bSigma_{12} \bSigma_{21} \bSigma_{11}^{-1}$ is positive semidefinite and has the same rank as $\bSigma_{12}$.  Since $\bSigma_{12} \neq \bzero$ then that rank is at least $1$, so at least one diagonal entry of $\bSigma_{11}^{-1} \bSigma_{12} \bSigma_{21} \bSigma_{11}^{-1}$ is positive; without loss of generality, we assume that the first diagonal entry, $(\bSigma_{11}^{-1} \bSigma_{12} \bSigma_{21} \bSigma_{11}^{-1})_{11}$, is positive.  Letting $\ee_1 = (1,0,\ldots,0)^\transp$, we obtain
$$
\|\bSigma_{21} \bSigma_{11}^{-1}\,\ee_1 \|^2 = (\bSigma_{21} \bSigma_{11}^{-1}\, \ee_1)^\transp \bSigma_{21} \bSigma_{11}^{-1}\, \ee_1 = (\bSigma_{11}^{-1} \bSigma_{12} \bSigma_{21} \bSigma_{11}^{-1})_{11} > 0;
$$
consequently, $\|v \bSigma_{21} \bSigma_{11}^{-1}\, \ee_1\| \to \infty$ as $v \to \infty$.

By \eqref{quadraticdecomposition},
$$
Q(v\ee_1,\Eta) =
\begin{pmatrix}v\ee_1 \\ \Eta\end{pmatrix}^\transp \bSigma^{-1} \begin{pmatrix}v\ee_1 \\ \Eta\end{pmatrix}
= v^2 \begin{pmatrix}\ee_1 \\ v^{-1}\Eta\end{pmatrix}^\transp \bSigma^{-1} \begin{pmatrix}\ee_1 \\ v^{-1}\Eta\end{pmatrix};
$$
therefore, as $v \to \infty$, we obtain $Q(v\ee_1,\Eta) \sim \kappa^2 v^2$ where
$$
\kappa^2 = \begin{pmatrix}\ee_1 \\ \bzero\end{pmatrix}^\transp \bSigma^{-1} \begin{pmatrix}\ee_1 \\ \bzero\end{pmatrix} = (\bSigma^{-1})_{11},
$$
the $(1,1)$th entry of $\bSigma^{-1}$.  Letting $v \to \infty$ in \eqref{gprime_condition}, we obtain
\begin{align}
\label{limitgprime}
\cc_1 &= \lim_{v\rightarrow \infty} \frac{g'(\kappa^2 v^2)}{g(\kappa^2 v^2)} \big(v \, \bSigma_{21} \bSigma_{11}^{-1}\ee_1 + \cc_2 \big) \nonumber \\
&= \lim_{v\rightarrow \infty} \frac{v g'(v^2)}{g(v^2)} \big(\kappa^{-1} \, \bSigma_{21} \bSigma_{11}^{-1}\ee_1 + v^{-1}\cc_2 \big).
\end{align}

By the regularity condition (\ref{R3}), $v g'(v^2)/g(v^2)$ tends to zero or diverges as $v \to \infty$.  If $v g'(v^2)/g(v^2) \to 0$ then the right-hand side of Equation (\ref{limitgprime}) tends to zero as $v \to \infty$, so we obtain $\cc_1 = \bzero$, which contradicts the fact that $\cc_1 \neq \bzero$.  On the other hand, if $v g'(v^2)/g(v^2)$ diverges as $v \to \infty$, then the right-hand side of Equation (\ref{limitgprime}) diverges, which contradicts the fact that $\|\cc_1\| < \infty$.  Since the assumption that $\bSigma_{12} \neq \bzero$ leads in either case to a contradiction then it follows that $\bSigma_{12} = \bzero$.

Since $\bSigma_{12} = \bzero$ then Equation \eqref{gprime_condition} reduces to
$$
\frac{g'(Q(\ww_1,\Eta))}{g(Q(\ww_1,\Eta))} \cc_2 = \cc_1;
$$
equivalently, $g'(t) = c_3 g(t)$, hence $g(t) = c_3 \exp(-c_4 t)$, for some constants $c_3$ and $c_4$.  Therefore, $\WW$ has a truncated multivariate normal distribution with $\bSigma_{12} = \bzero$.
\end{proof}

\section{Testing the independence of the components of a truncated multivariate normal vector}
\label{sec:application}

As an application of our results, we perform a likelihood ratio test for independence between $W_1$ and $W_2$, the components of a bivariate truncated normal random vector.  For $j,k=1,2$, denote by $\sigma_{jk}$ the $(j,k)$th element of $\bSigma$; let $\sigma_j = \sigma_{jj}^{1/2}$; and set $\rho = {\sigma_{12}}/(\sigma_{1}\sigma_{2})$.  By Theorem \ref{independence_trnorm}, testing for independence between $W_1$ and $W_2$ is equivalent to testing the null hypothesis, $H_{0}:\rho=0$, {\it vs.}~the alternative hypothesis, $H_a:\rho \ne 0$.  For illustrative purposes, we apply the test to a data set, considered by Cohen \cite[p.~192]{Cohen}, consisting of the entrance examination scores, $W_1$, and subsequent course averages, $W_2$, achieved by $n=529$ university students.  The data are viewed as generated randomly from a bivariate truncated normal distribution, with the cutoff value for $W_1$ being $159.5$, the minimum qualifying score on the entrance examination, and with the cutoff value for $W_2$ being $c_{2}=0$ since all course averages are nonnegative, respectively.  With these constraints, $n=517$ students were admitted.

Corresponding to $H_a$, we denote the unrestricted (or alternative) parameter space by $\bTheta = \{\btheta=(\mu_{1},\mu_{2},\sigma_{1},\sigma_{2},\rho)': \mu_{1} \in \Re, \mu_{2} \in \Re, \sigma_{1} > 0,\sigma_{2} > 0, -1 < \rho < 1\}$; wherever necessary, we may also denote the respective individual components of $\btheta$ by $\theta_i$, $i=1,\ldots,5$.  The restricted (or null) parameter space, as determined by $H_0$, then is $\bTheta_{0}=\{\mu_{1},\mu_{2}\in \Re, \sigma_{1},\sigma_{2}>0,\rho=0\}$, and the likelihood ratio test statistic for testing $H_0$ {\it vs.}~$H_a$ is
$$
\Lambda = \frac{\sup_{\btheta\in \bTheta_{0}}L(\btheta)}{\sup_{\btheta\in \bTheta}L(\btheta)}.
$$

For $\ww = (w_1,w_2)$, we write the joint probability density function of $\WW$ in the form
\begin{equation}
\label{bvmt}
f(\ww;\btheta)=\frac{C}{2\pi\sigma_{1}\sigma_{2} (1-\rho^{2})^{1/2}} \, \exp\Big(-\frac{1}{2(1-\rho^2)} Q(\ww;\btheta)\Big),
\end{equation}
$w_1 \ge c_1$, $w_2 \ge c_2$, where
$$
Q(\ww;\btheta) = \frac{(w_{1}-\mu_{1})^2}{\sigma_{1}^2}
-\frac{2\rho(w_{1}-\mu_{1})(w_{2}-\mu_{2})}{\sigma_{1}\sigma_{2}}+\frac{(w_{2}-\mu_{2})^2}{\sigma_{2}^2},
$$
and
\begin{equation*}
%\label{const}
C^{-1}:=\frac{1}{2\pi\sigma_{1}\sigma_{2} (1-\rho^{2})^{1/2}}
\int_{c_{1}}^{\infty}\int_{c_{2}}^{\infty}\exp\Big(-\frac{1}{2(1-\rho^2)}Q(\ww;\btheta)\Big) \ds w_{1} \ds w_{2}
\end{equation*}
is the normalizing constant.

For a random sample $(W_{1,1},W_{2,1})',\ldots,(W_{1,n},W_{2,n})'$ from the distribution \eqref{bvmt}, the corresponding log-likelihood function can be written, up to additive constants that do not depend on the parameter $\btheta$, as
%\begin{equation}
%\label{log_likelihood}
\begin{align*}
\log L(\btheta) &= -n \big[\log C-\log \sigma_{1}-\log \sigma_{2}-\tfrac{1}{2}\log(1-\rho^{2})\big] \\
& \quad + \frac{1}{2(1-\rho^{2})} \sum_{i=1}^2 \frac{1}{\sigma_i^2} \sum_{j=1}^{n}(w_{i,j}-\mu_{i})^{2} \\
& \quad - \frac{\rho}{\sigma_{1}\sigma_{2}(1-\rho^{2})}\sum_{j=1}^{n}(w_{1,j}-\mu_{1})(w_{2,j}-\mu_{2}).
\end{align*}

The asymptotic null distribution of the likelihood ratio statistic is derived from a classical theorem of Wilks \cite{Wilks} (\textit{cf.}, Casella and Berger \cite[pp.~489,~516]{Casella}, Hogg, {\it et al.} \cite[p.~361]{Hogg}).  First, we verify that the regularity conditions underlying Wilks' theorem are valid for the truncated normal distribution:

\begin{itemize}[leftmargin=15pt]
\setlength\itemsep{-0.1em}

\item[\namedlabel{1.}{1.}]
{\it The density $f(\ww;\btheta)$ is identifiable, i.e., if $f(\ww;\btheta_1) = f(\ww;\btheta_2)$ for all $\ww \ge \cc$ then $\btheta_{1}=\btheta_{2}$}: To prove this result, note that
\begin{align*}
Q(\ww;\btheta) \equiv \frac{w_{1}^{2}}{\sigma_{1}^{2}}+\frac{w_{2}^{2}}{\sigma_{2}^{2}}-2w_{1}\left(\frac{\mu_{1}}{\sigma_{1}^{2}}-\frac{\rho\mu_{2}}{\sigma_{1}\sigma_{2}}\right) - 2w_{2}\left(\frac{\mu_{2}}{\sigma_{2}^{2}}-\frac{\rho\mu_{1}}{\sigma_{1}\sigma_{2}}\right) - \frac{2\rho w_{1}w_{2}}{\sigma_{1}\sigma_{2}};
\end{align*}
then it follows from \eqref{bvmt} that the truncated bivariate normal distribution is an exponential family with {\it natural} (or {\it canonical}) sufficient statistic
$$
\VV=(w_{1},-w_{1}^{2},w_{2},-w_{2}^{2},w_{1}w_{2})'
$$
and corresponding {\it canonical} parameter vector
\begin{equation*}
%\label{min}
\TT=\left(\frac{1}{\sigma_{1}^{2}},\frac{1}{\sigma_{2}^{2}},2\left(\frac{\mu_{1}}{\sigma_{1}^{2}}-\frac{\rho\mu_{2}}{\sigma_{1}\sigma_{2}}\right),2\left(\frac{\mu_{2}}{\sigma_{2}^{2}}-\frac{\rho\mu_{1}}{\sigma_{1}\sigma_{2}}\right),\frac{2\rho}{\sigma_{1}\sigma_{2}}\right)'.
\end{equation*}
It is now evident that the components of the natural sufficient statistic and of the canonical parameter vector are linearly independent over $\Re^5$.  Further, the exponential family is {\it minimal}, meaning that it is five-dimensional and cannot be reduced to a lower-dimensional model.  Consequently, by Barndorff-Nielsen \cite[pp.\ 112--113, Lemma 8.1 and Corollary 8.2]{Barndorff}, the model \eqref{bvmt} is identifiable.

\item[\namedlabel{2.}{2.}]
{\it The support of the distribution remains the same for all values of $\btheta$}:  This condition is clearly satisfied since the density $f(\ww;\btheta)$ has support $(c_{1},\infty)\times (c_{2},\infty)$, which does not depend on $\btheta$.

\item[\namedlabel{3.}{3.}]
{\it There exists an open subset $\bOmega_{0} \subset \bTheta$ such that the ``true value'' of the parameter $\btheta$ is in $\bOmega_{0}$, and all third-order partial derivatives of $f(\ww;\btheta)$ with respect to $\ww$ exist for all $\btheta\in \bOmega_{0}$}:  This condition is satisfied since $\bTheta$ is an open subset of $\Re^{5}$ and we can construct $\bOmega_{0}$ consisting of the union of sufficiently small open univariate balls around the true value of each of the parameters $\theta_1,\ldots,\theta_5$.  Further, the differentiability property follows from \eqref{bvmt}.

\item[\namedlabel{4.}{4.}]
{\it The integral $\int f(\ww|\btheta) \ds \ww$ is twice-differentiable with respect to $\btheta$}:  According to the usual Leibniz rule, partial derivatives and integrals may be interchanged whenever the same derivatives of the density function $f(\ww|\btheta)$ are continuous and integrable for all $\btheta \in \bTheta$ and all $w_1 \ge c_1$, $w_2 \ge c_2$; see Burkill and Burkill \cite[p.~289, Theorem 8.72]{Burkill}.  In the case of \eqref{bvmt}, the conditions for Leibniz' rule follows from the finiteness of the moments of any positive order for that distribution.   We also note that Barndorff-Nielsen \cite[p.~114, Theorem 8.1]{Barndorff} shows that differentiation with respect to $\theta$, to any order, of the integral is allowed under the integral sign.

\item[\namedlabel{5.}{5.}]
{\it The information matrix $I(\btheta)$ of the density function $f(\ww;\btheta)$ is positive definite}: As shown earlier, $f(\ww;\btheta)$ is a non-curved minimal exponential model.  By a well-known result for exponential families \cite[Section 9.3]{Barndorff}, the covariance matrix of $\VV$, denoted by $\Cov(\VV)$, is a full-rank matrix and therefore is positive definite. Since the information matrix is
$$
I(\btheta)=\Big(\frac{\partial \TT}{\partial \btheta}\Big)'\Cov(\VV)\frac{\partial \TT}{\partial \btheta},
$$
then it follows that $I(\btheta)$ also is of full rank.

\item[\namedlabel{6.}{6.}]
{\it All third-order partial derivatives of $\log f(\ww;\btheta)$ are bounded by functions of $\ww$ that have finite expectations}:  By straightforward differentiation with respect to $\theta_{j}$, $\theta_{k}$, and $\theta_{l}$, $1\le j,k,l\le 5$, we obtain
$$
\frac{\partial^3}{\partial \theta_j \partial \theta_k \partial \theta_l} \log f(\ww;\btheta) \le P_{jkl}(\ww),
$$
where $P_{jkl}(\ww)$ is a polynomial in $w_1, w_2$.  Since all polynomial moments of the truncated bivariate normal distribution are finite then $\E\, P_{jkl}(\WW) < \infty$ for all $j,k,l$.
\end{itemize}

Having shown that the regularity conditions underpinning Wilks' theorem are satisfied in our setting, we deduce that, under $H_0$, $-2 \log \Lambda \rightarrow \chi^{2}_{1}$ in distribution as $n \rightarrow \infty$.

To apply to the data of Cohen \cite[{\it loc. cit.}]{Cohen} the likelihood ratio statistic for testing $H_0$ {\it vs.}~$H_a$, we calculated the maximum likelihood estimates of the parameters of the bivariate truncated normal distribution using the \textsl{R} package of Wilhelm and Manjunath \cite{Wilhelm}; alternatively, the calculations can be done using the procedures described by Cohen \cite[pp.~186--190]{Cohen}.  We obtained from the \textsl{R} package the estimates,
$$
\widehat{\mu}_1=164.19,  \quad \widehat{\mu}_2=77.195, \quad \widehat{\sigma}_{1}=3.059, \quad \widehat{\sigma}_{2}=5.459, \quad \widehat{\rho}=0.431.
$$
The resulting observed value of the test statistic $-2\log\Lambda$ was $84.905$, and the corresponding P-value was found to be approximately $3.130423*10^{-20}$.  Consequently, the null hypothesis $H_0$ is rejected at any practical level of significance.

\section{Conclusions}
\label{sec:conclusions}

We have shown that the mutual independence of the components of a multivariate truncated elliptical distribution is equivalent to $\bSigma_{12}=\bzero$ subject to additional regularity conditions on the generator function $g$.  If these regularity conditions are satisfied then they imply that the underlying distribution is the truncated multivariate normal distribution. These results suggest two problems for future research. The first problem concerns the existence of multivariate truncated elliptical distributions, other than the truncated normal, for which $\bSigma_{12}=\bzero$ is equivalent to independence of its components. This problem leads naturally to a search for regularity conditions weaker than the ones that we have used in Corollary \ref{ell}.  The second direction is to characterize the property of uncorrelatedness for the multivariate truncated elliptical distributions; explicitly, the goal will be to obtain explicit criteria, in terms of the correlation matrix of the underlying multivariate elliptical distribution and its generator function, that are equivalent to zero correlation between components of its truncated analogs, $\WW_1$ and $\WW_2$.   We plan to study both of these directions in future research.

\end{document}